\documentclass[11pt]{amsart}
\usepackage{amsmath,amsfonts,amssymb,amsthm,amscd,hyperref,soul}%,showkeys}
\usepackage{dsfont}
\usepackage[english]{babel}
\usepackage[applemac]{inputenc}
\usepackage[mathscr]{eucal}
\usepackage{graphicx,url}
\usepackage{appendix}
\usepackage{color}
\newtheorem{theorem}{Theorem}[section]

\newtheorem{lemma}[theorem]{Lemma}
\newtheorem{corollary}[theorem]{Corollary}

\newtheorem{remark}[theorem]{Remark}
\theoremstyle{definition}

\newcommand{\R}{\mathbb{R}}
\newcommand{\C}{\mathbb{C}}

\newcommand{\ds}{\partial_\sigma}
\newcommand{\re}{\Re \text{e}}

\title[]{Collision of almost parallel\\ vortex filaments}
\author[V. Banica]{Valeria Banica}
\address[V. Banica]
{Laboratoire de Math\'ematiques et de Mod\'elisation d'\'Evry (UMR 8071)\\ 
Universit\'e d'\'Evry, 23 Bd. de France, 91037 Evry\\ 
France, Valeria.Banica@univ-evry.fr}

\author[E. Faou]{Erwan Faou}
\address[E. Faou]
{INRIA-Rennes Bretagne Atlantique and IRMAR (Universit\'e de Rennes I)\\ 
France, Erwan.Faou@inria.fr} 

\author[E. Miot]{Evelyne Miot}
\address[E. Miot]
{Centre de Math\'ematiques Laurent Schwartz (UMR 7640) \\ 
\'Ecole Polytechnique 91128 Palaiseau\\ 
France, Evelyne.Miot@math.polytechnique.fr} 

\thanks{First and last authors are partially supported by the French ANR project SchEq ANR-12-JS-0005-01. The second author is supported by the ERC starting grant GEOPARDI No. 279389. The last author is partially supported by the French ANR project GEODISP ANR-12-BS01-0015-01.}

\begin{document}
\date{June 8, 2015}

\maketitle

\begin{abstract}
We investigate the occurrence of collisions in the evolution of vortex filaments  through a system introduced by Klein, Majda and Damodaran  \cite{KlMaDa} and Zakharov \cite{Zh1,Zh2}. We first establish rigorously the existence of a pair of almost parallel vortex filaments, with opposite circulation, colliding at some point in finite time. The collision mechanism is based on the one of the self-similar solutions of the model, described in \cite{BFM}. In the second part of this paper we extend this construction to the case of an arbitrary number of filaments, with polygonial symmetry, that are perturbations of a configuration of parallel vortex filaments forming a polygon, with or without its center, rotating with constant angular velocity.
\end{abstract}

\section{Introduction}

We consider the system introduced by Klein, Majda and Damodaran  \cite{KlMaDa} and Zakharov \cite{Zh1,Zh2} to describe the evolution of $N$ almost parallel vortex filaments in a three-dimensional 
{incompressible} fluid.  According to this model, the vortex filaments are curves parametrized by 
 $$(\re( \Psi_j(t,\sigma)), \Im\text{m} (\Psi_j(t,\sigma)),\sigma),\quad \sigma\in \R,\quad 1\leq j\leq N,$$
where $\Psi_j:\R\times \R\to \mathbb{C}$. The dynamics of the set of curves is governed by an Hamiltonian system of one-dimensional Schr\"odinger  equations with vortex type interaction for the maps $\Psi_j$:
\begin{equation}\label{KMD}
\begin{cases}
\displaystyle 
i\partial_t \Psi_j+\alpha_j\Gamma_j \partial_\sigma^2\Psi_j+\sum_{k\neq j} \Gamma_k \frac{\Psi_j-\Psi_k}{|\Psi_j-\Psi_k|^2}=0,\quad 1\leq j\leq N,\\
\Psi_j(0,\sigma)=\Psi_{j,0}(\sigma). 
\end{cases}
\end{equation}Moreover, the filaments are parallel at infinity, which means that there exist $z_j(t)\in \C$ such that
\begin{equation}\label{cond:parallel}\Psi_j(t,\sigma)\to z_j(t),\quad \text{as }\sigma \to \pm \infty.\end{equation}

In \eqref{KMD}, $\alpha_j\in\mathbb R^+$ denotes a parameter related to the core structure  and $\Gamma_j\in \R^\ast$ the circulation of the $j$th filament. Throughout the paper, we will assume that $\alpha_j=1$ for all $1\leq j\leq N$. 

Particular solutions of \eqref{KMD} - \eqref{cond:parallel} are given by the collections of parallel filaments, for which  $\Psi_j(t,\sigma)=X_j(t,\sigma)=(z_j(t),\sigma)$, where $(z_j)$ is a solution of the 2-D point vortex system
\begin{equation}\label{point-vortex}
\displaystyle 
i\dot{z}_j+\sum_{k\neq j} \Gamma_k \frac{z_j-z_k}{|z_j-z_k|^2}=0,\quad 1\leq j\leq N.
\end{equation}
%{\color{blue}The normalized energy associated to the system \eqref{KMD} is
%$$\mathcal{H}=\frac{1}{2}\sum_{j}\alpha_j \Gamma_j^2 \int \left|\partial_\sigma \Psi_j\right|^2\,d\sigma-\frac{1}{2}\int \sum_{j\neq k} \Gamma_j \Gamma_k  \ln\left(\frac{|\Psi_{j}-\Psi_k|^2}{|X_{j}-X_k|^2}\right)\,d\sigma.$$}
The Cauchy theory for \eqref{KMD} with the condition \eqref{cond:parallel}, under the assumption that the curves $(\Psi_j(t,\sigma),\sigma)$ are perturbations of the lines $(z_j(t),\sigma)$,  has been studied by Klein, Majda and Damodaran \cite{KlMaDa}, by Kenig, Ponce and Vega \cite{KePoVe} and more recently by Banica and Miot \cite{BaMi, BaMi2} and Banica, Faou and Miot \cite{BFM}.  In these works, the maximal time of existence corresponds to the first occurence of a collision between two or more filaments. We also refer to the results  by Lions and Majda \cite{LiMa} on global existence of weak, space-periodic solutions. The aim of the present paper is to construct a solution to \eqref{KMD} with the condition \eqref{cond:parallel} such that the filaments exhibit a collision in finite time, in two particular situations that are described below. 
\subsection{The anti-parallel pair of filaments}
We will first consider pairs of filaments, namely $N=2$, with opposite circulations $$\Gamma_1=-\Gamma_2.$$ By rescaling in space and time, there is no restriction to rewrite system \eqref{KMD} as
\begin{equation}\label{KMD2same}
\begin{cases}
\displaystyle i\partial_t \Psi_1+\partial_\sigma^2\Psi_1-2\frac{\Psi_1-\Psi_2}{|\Psi_1-\Psi_2|^2}=0,\\
\displaystyle i\partial_t \Psi_2-\partial_\sigma^2\Psi_2-2\frac{\Psi_1-\Psi_2}{|\Psi_1-\Psi_2|^2}=0.\\
\end{cases}
\end{equation}
Reconnection - or collision -, that is existence of a point $(t,\sigma)$ where the filaments collide, occurs when $\Psi_1(t,\sigma)-\Psi_2(t,\sigma)=0$.
We note that we have as a particular solution of \eqref{KMD} the so-called anti-parallel vortex filament pair $(X_1,X_2)$ with the parallel filaments $X_j(t,\sigma)=(z_j(t),\sigma)$ given by
$$z_1(t)=-it+1,\quad z_2(t)=-{it}-1.$$
Since Crow's work  \cite{Cr} in the seventies, examples of vortex filament reconnection are searched as perturbations of the anti-parallel vortex filament pair. Most of these results are numerical, and based on initial perturbations in link with the unstable mode of the linearized equation. In this paper we shall prove the existence of a perturbation of the anti-parallel vortex filament pair reconnecting in finite time through \eqref{KMD2same}. The collision mechanism will be based on the one of the self-similar solutions of \eqref{KMD2same}, that we describe next.

We consider as in \cite{Zh1,Zh2,MaBe} an initial configuration of two filaments satisfying the following symmetry:
$$\Psi_1(0)=-\overline{\Psi}_2(0).$$
Since $(-\overline{\Psi}_2,-\overline{\Psi}_1)$ is still a solution of  \eqref{KMD2same} with same initial datum, this kind of symmetry is conserved as long as the solution $(\Psi_1,\Psi_2)$ to \eqref{KMD2same} exists:
$$\Psi_1(t)=-\overline{\Psi}_2(t).$$
This yields $\Psi_1-\Psi_2=2\re(\Psi_1)$ and so system \eqref{KMD2same} reduces to only one equation:
\begin{equation}\label{eqpsi}
 i\partial_t \Psi_1+\partial_\sigma^2\Psi_1-\frac{1}{\re(\Psi_1)}=0.
\end{equation}
 
Constructing filaments exhibiting a collision in finite time reduces to finding solutions to the Schr\"odinger-like equation \eqref{eqpsi} vanishing in finite time at some point, starting with non-vanishing initial data. This requires the non-standard study of pointwise control of nonlinear solutions. This kind of issue has been considered by Merle and Zaag \cite{MeZa} for a model related to the heat equation. However the approach developed in \cite{MeZa} does not seem to be tractable to our case. For more details see the discussion at the end of \S2 in \cite{BFM}.\\

As suggested in \cite{Zh1,Zh2,KlMaDa} a natural candidate for a solution of \eqref{eqpsi} with collision in finite time should be self-similar, namely
$$\Psi_1(t,\sigma)=\sqrt{t}\:u\left(\frac{\sigma}{\sqrt{t}}\right),$$
with the asymptotics
\begin{equation}\label{infcond}
u(x)\sim \alpha |x|,\quad |x|\to+\infty,
\end{equation}
where $\alpha \in \C$.
%(here we denote $x=\sigma/\sqrt{t}$). 
In view of \eqref{eqpsi}, the equation for the profile $u(x)$ is 
\begin{equation}\label{equ}
i(u-xu')+2 u''-\frac{2}{\re(u)}=0.
\end{equation}
We rewrite equation \eqref{equ} as
\begin{equation}\label{syst:self-sim}
\begin{cases}
\displaystyle v' = \frac{1}{2}{i x v} - \frac{x}{ \re(u) }\\
\displaystyle u - x u' = v.
\end{cases}
\end{equation}

We recall our result  \cite[Theo. 3.1]{BFM}.
\begin{theorem}[\cite{BFM}]\label{thm:ak}There exists $K>1$ such that the following holds. 
Let $\alpha\in \C$  with $\re(\alpha)>K$ and 
$$E=\left\{w\in C^1(\R),\:w(0)=w'(0)=0,\:\: w\text{  even  },\:\: \|w\|_{L^\infty} + \||x|^{-1}w'\|_{L^\infty}\leq \frac{\re(\alpha)}{4}\right\}.$$
There exists a unique $v \in 1+ E$ such that the couple $(u,v)$, with $u$ defined by
\begin{equation}\label{couplage}
u(x) = 1+|x|\left( \alpha + \int_{|x|}^{+\infty} \frac{v(z)-1}{z^2} \,dz\right), \quad \forall x\neq 0,
\end{equation}
with 
$$u(0)=1,$$
is a solution of the system \eqref{syst:self-sim} satisfying the condition \eqref{infcond}. Moreover,
\begin{equation}\label{poscond}
\re ((u(x)))>0, \quad \forall x\in\mathbb R.
\end{equation}
Finally, $u$ is a Lipschitz function on $\R$ with $u'\in C(\R\setminus \{0\})$.
\end{theorem}
In view of the regularity assumption on $u$, we infer that Theorem \ref{thm:ak} provides  $(\Psi_1,\Psi_2)$ satisfying the equation \eqref{KMD2same} in a strong sense on $\mathbb R_+^\ast\times \mathbb R^*$ and in mild sense on $\mathbb R_+^\ast \times \mathbb R$.  Since
\begin{equation*}
(\Psi_1-\Psi_2)(t,\sigma) = 2\sqrt{t}\left(\re \:u\right)  \left(\frac{\sigma}{\sqrt{t}}\right),
\end{equation*}
the property \eqref{poscond} implies that for $t> 0$  no collision occurs. Moreover, the condition $u(0)=1$ implies that $|\Psi_1(t,0)-\Psi_2(t,0)|=2\sqrt{t}$. Therefore collision occurs at time $t=0$ and at position $\sigma=0$. 

Unfortunately, the self-similar solutions constructed in Theorem \ref{thm:ak} do not enter the setting of the  model \eqref{KMD} - \eqref{cond:parallel}. Indeed, given the asymptotics \eqref{infcond}, the two filaments parametrized by $\Psi_1$ and $\Psi_2$ are not parallel at infinity. Still, having at one's disposal such configurations with a precise description of their behavior will enable us to construct colliding solutions of \eqref{KMD} that are parallel at infinity. 
Our main results are the following.
\begin{theorem} There exists $K_0\geq K>1$ such that the following holds. Let $\alpha>K_0$.
\label{thm:main}
Let $$H(t,\sigma)=\sqrt{t}u\left(\frac{\sigma}{\sqrt{t}}\right),$$ where $u$ is given in Theorem \ref{thm:ak} for this value of $\alpha$.
Then there exists $t_0\in (0,1)$, depending on $\alpha$, and there exists a solution $\Psi_1:(0,t_0]\times\R\to \mathbb{C}$ to the equation \eqref{eqpsi}, which has the form
\begin{equation}\label{main:ansatz}
\Psi_1(t,\sigma)=-it+r(t,\sigma)+\frac{H(t,\sigma)}{1+\psi(\alpha |\sigma|)},
\end{equation}
where $\psi(\tau)=\tau \varphi(\tau)$, with $\varphi:\R_+\to [0,1]$ 
a smooth, cut-off function satisfying $\varphi=0$ on $[0,1]$ and $\varphi=1$ on $[2,+\infty)$, and where $r\in C([0,t_0], H^1(\R))$ satisfies  $\|r(t)\|_{L^\infty} < \sqrt{t}/4$ as $t\to 0$. Moreover, we have
\begin{equation}\label{chichi}
\re \left(\Psi_1(t,\sigma)\right)\geq
\begin{cases}\displaystyle  \frac{1}{4}(\sqrt{t}+\alpha |\sigma|)\quad \text{if}\quad \alpha |\sigma|\leq \frac 12,\\\,\\
\displaystyle \frac{1}{32}\quad \text{if}\quad \alpha |\sigma| \geq \frac 12.
\end{cases}
\end{equation}
In particular, 
$$\re (\Psi_1(t,\sigma))>0,\quad \forall (t,\sigma)\in (0,t_0]\times \R,\quad \re(\Psi_1(0,0))=0,$$
and for all $t\in (0,t_0]$
$$\Psi_1(t,\sigma)\to -it+1,\quad \text{as } \sigma \to \pm \infty.$$
\end{theorem}
\begin{remark} For a given $\alpha$ and cut-off function $\psi$, the map $\Psi_1$ constructed in Theorem \ref{thm:main} is the unique mild solution of \eqref{eqpsi}  of the form \eqref{main:ansatz} in the functional space considered for the perturbation $r(t,\sigma)$ (see \eqref{def:X-1}, \eqref{def:X-2})), that is 
\begin{multline*}
\Psi_1(t,\sigma)=-it+1+e^{it\partial_\sigma^2}\left(\frac{\alpha |\sigma|}{1+\psi(\alpha |\sigma|)}-1\right)\\[2ex] -i\int_0^te^{i(t-s)\partial_\sigma^2}\left(\frac{1}{\re (\Psi_1(s,\sigma))}-1\right)\,ds.
\end{multline*}
Note that $$\frac{\alpha |\sigma|}{1+\psi(\alpha |\sigma|)}-1=\frac{H(0,\sigma)}{1+\psi(\alpha |\sigma|)}-1\in L^2(\R),$$
and that, in view of the estimates established in Lemmas \ref{lemma:H} and \ref{lemma:H-r},
$$\frac{1}{\re (\Psi_1)}-1\in L^1((0,t_0),L^2(\R)).$$
 \end{remark}

\begin{corollary}
With the notations of the previous theorem, the filaments $(\Psi_1, \Psi_2)$ with $\Psi_2=-\overline{\Psi}_1$ are solutions to \eqref{KMD2same}, satisfying the parallelism condition at infinity \eqref{cond:parallel}, and colliding at time $t=0$ at position $\sigma=0$.

%In particular, one has $$\re (\Psi_1(t,\sigma))>0,\quad \forall (t,\sigma)\in (0,t_0]\times \R,$$ and $$\Psi_1(t,0)=-it+r(t,0)+\sqrt{t}.$$ Since  $$(\Psi_1-\Psi_2)(t,\sigma)=(2\re (\Psi_1(t,\sigma)),0,0),$$ collision occurs only at time $t=0$ and at position $\sigma=0$. Since for all $t\in (0,t_0]$ $$\Psi(t,\sigma)\to -it+1,\quad \text{as } \sigma \to \pm \infty,$$ the condition \eqref{cond:parallel} is satisfied, namely the filaments $\Psi_1$ and $\Psi_2$ are parallel at infinity.
\end{corollary}

%\medskip
%We have the following comments on Theorem \ref{thm:main}.
%First, the corner-type self-similar mechanism of reconnection is coherent with the former numerical experiments on this problem and generically for singularity formation in finite time. \textcolor{blue}{Reference??}

\medskip

Note that at the collision time $t=0$ we have $\Psi(0,\sigma)=\alpha |\sigma|/(1+\psi(\alpha |\sigma|))$. 
By varying the function $\psi$ we obtain a whole family of collision scenarios for two filaments, with the commun specificity that a corner is formed for each filament at $\sigma=0$. A similar behavior occurs in the setting of the binormal flow equation, which governs the evolution of a single vortex filament. More precisely, Banica and Vega displayed in \cite{BV} a wide class of solutions generating a corner in finite time through the binormal flow equation, based on the description of the self-similar solutions provided by Guti\'errez, Rivas and Vega in \cite{GRV}.

%The system \eqref{KMD} is derived for fluids, by taking into account one one hand the influence of one filament on itself, and on the other hand the interaction between filaments. The first part is modeled by the linear Schr\"odinger equation, as an approximation of the more involved model for a single filament given by LIA. We recall that LIA is also the model communly used for one vortex filament dynamics in superfluids (\cite{LD},\cite{LRT},\cite{Bu}). At the experimantal level, it is only very recently that Fonda-Meichle-Ouellette-Hormoz-Lathrop \cite{Fonda} dispayed Kelvin waves in vortex reconnection in superfluid Helium, which would correspond to the self-similar solutions of the model. It would be interesting to see if the approach made here could be extended to models in superfluids concerning almost parallel vortex filaments. 

\subsection{The case of several  filaments with polygonial symmetry}
Our next purpose is to extend the example of collision built for the pair of filaments to the case of an arbitrary number $N+1$ of  filaments associated with the functions $\Psi_0(t,\sigma)$, $\Psi_j(t,\sigma)$, $1 \leq j \leq N$ satisfying \eqref{KMD} with 
$$\Gamma_j=\Gamma,\quad \forall\, j, \quad  1\leq j\leq N, \quad \mbox{and} \quad \Gamma_0\in\mathbb R.$$
Without loss of generality we will assume that
$\Gamma=1.$
Setting
$$\omega=\frac{(N-1)}{2}+\Gamma_0,$$
we note that a particular solution of \eqref{KMD} is given by the collection of parallel filaments $X_j(t,\sigma)=(z_j(t),\sigma)$ with polygonial symmetry, obtained from the solution of the 2-D point vortex system forming a regular rotating polygon, with or without its center, with angular velocity $\omega$: 
\begin{equation}\label{def:sol-exacte} z_j(t)=e^{i(j-1)\frac{2\pi}{N}  }e^{i\omega t}z_1(0),\quad 1\leq j\leq N,\quad z_0(t)=0.\end{equation}
We focus on collections of filaments with the same symmetry as above, namely
\begin{equation}\label{hyp:sym}\Psi_{j}(t,\sigma)=e^{i(j-1)\frac{2\pi}{N}  }\Psi_1(t,\sigma),\quad \forall 1\leq j\leq N,\quad \Psi_0(t,\sigma)=0,\quad \forall \sigma \in \R.\end{equation}
This symmetry assumption is preserved in time. The system \eqref{KMD} is then written as a single equation
\begin{equation}\label{eqpsi-pol}
 i\partial_t \Psi_1+\partial_\sigma^2\Psi_1+\omega\frac{\Psi_1}{|\Psi_1|^2}=0.
\end{equation}
In the present case, reconnection of all filaments at the same place corresponds to a zero of $\Psi_1$.

\medskip

As for the pair of filaments, we first investigate the existence of a  self-similar solution of \eqref{eqpsi-pol} 
$$\Psi_1(t,\sigma)=\sqrt{t}\:u\left(\frac{\sigma}{\sqrt{t}}\right),$$
with the asymptotic behavior
\begin{equation}\label{infcond-pol}
u(x)\sim \alpha |x|,\quad |x|\to+\infty.
\end{equation}
The equation for the profile $u(x)$ now is 
\begin{equation}\label{equ-pol}
i(u-xu')+2u''+2\omega\frac{ u}{|u|^2}=0
\end{equation}
which we decouple as
\begin{equation}\label{syst:self-sim-pol}
\begin{cases}
\displaystyle v' = \frac{1}{2}i x v + \omega x \frac{ u}{|u|^2}\\
\displaystyle u - x u' = v.
\end{cases}
\end{equation}

We will establish the analogous result of Theorem \ref{thm:ak}.
\begin{theorem}\label{thm:ak-pol}There exists $\tilde{K}>1$, depending on $\omega$, such that the following holds. 
Let $\alpha \in \C$  with $\re(\alpha)>\tilde{K}$ and such that$|\alpha|\leq \re(\alpha)^2$. Let
$$E=\left\{w\in C^1(\R),\:w(0)=w'(0)=0,\:\: w\text{  even  },\:\: \|w\|_{L^\infty} + \||x|^{-1}w'\|_{L^\infty}\leq \frac{\re(\alpha)}{4}\right\}.$$
There exists a unique $v \in 1+ E$ such that the couple $(u,v)$, with $u$ defined by
\begin{equation}\label{couplage}
u(x) = 1+|x|\left( \alpha + \int_{|x|}^{+\infty} \frac{v(z)-1}{z^2} \,dz\right), \quad \forall x\neq 0,
\end{equation}
with $$u(0)=1,$$
is a solution of the system \eqref{syst:self-sim-pol} satisfying the condition \eqref{infcond-pol}. Moreover,
\begin{equation}\label{poscond-pol}
|u(x)|\geq 1,\quad \forall x\in\mathbb R,
\end{equation}
and $u$ is an even, Lipschitz function on $\R$ with $u'\in C(\R\setminus \{0\})$.
\end{theorem}

Our next aim is to establish the existence of a solution $\Psi_1$ to \eqref{eqpsi-pol}, satisfying \eqref{cond:parallel}, such that $\Psi_1(0,0)=0$. As for the anti-parallel pair of filaments, this will be performed by perturbing an exact parallel solution of \eqref{KMD} by the self-similar solution given by Theorem \ref{thm:ak-pol} after renormalization at infinity.
 More precisely, in view of \eqref{hyp:sym}, we will look for a solution having the form
\begin{equation}\label{cond:polygon}\Psi_j(t,\sigma)=z_j(t)\Phi(t,\sigma),\quad 0\leq j\leq N,\end{equation} where $(z_j)$ is the configuration given by \eqref{def:sol-exacte}, and where the perturbation $\Phi$ satisfies $$\Phi(t,\sigma)\to 1\quad \text{as } \sigma\to \pm \infty\quad \text{and  }\Phi(0,0)=0.$$ This ensures that the filaments are parallel at infinity and that they have the polygonial symmetry. Here, we consider a multiplicative perturbation of the exact parallel solution $X_j(t,\sigma)=(z_j(t),\sigma)$, while for anti-parallel pairs we dealt with an additive perturbation. Note also that, setting $z_1(0)=\rho e^{i\theta}$, we have
$$\Psi_1(t,\sigma)=\rho e^{i\theta + i \omega t} \Phi(t,\sigma).$$

The dynamics of symmetric configurations of vortex filaments  has been studied by Banica and Miot \cite{BaMi} through the analysis of the equation for the perturbation $\Phi$ 
\begin{equation}\label{BM}
i\partial_t \Phi+\partial_\sigma^2\Phi+\omega\frac{\Phi}{|\Phi|^2}(1-|\Phi|^2)=0.\end{equation} 
Collisions have been dispayed in \cite{BaMi} for the configurations yielded by the stationary polygon solution of the 2-D point vortex system \eqref{def:sol-exacte}. In particular, $\omega=0$ implies that such stationary configuration corresponds to the vertices of a regular polygon and its center with the non-trivial circulation $\Gamma_0=-\frac{N-1}2$. The case $\omega\neq 0$, namely the non-stationary rotating polygon solution of the 2-D point vortex system \eqref{def:sol-exacte}, with or without its center, was out of reach in \cite{BaMi}. In the present paper,  we are able to display collisions also in these cases. 

\medskip

Let us recall that in \cite{BaMi} it is proved that for all $\omega>0$, \eqref{BM} admits a unique global solution with finite and small energy, namely such that $$E(\Phi(t)):=\int_{\R} |\nabla \Phi(t)|^2\,dx + \omega  \int_{\R}\left(-\ln|\Phi(t)|^2+|\Phi(t)|^2-1)\right)\,dx\leq \eta,$$ 
where $\eta$ is an absolute constant.
By Gagliardo-Nirenberg inequality, the small energy assumption guarantees that $\sup_{t\in \R}\| |\Phi(t)|-1\|_{L^\infty}<1/2$; hence $\Phi$ does not vanish on $\R$.
Here, we will prove that for any $\omega \in \R$ there exists a finite energy solution of \eqref{BM} that vanishes at $(t,\sigma)=(0,0)$:
\begin{theorem} Let $\omega\in \R$. There exists $\tilde{K}_0\geq \tilde{K}>1$ such that the following holds. Let $\alpha>\tilde{K}_0$.
\label{thm:main-pol}
Let $$H(t,\sigma)=\sqrt{t}u\left(\frac{\sigma}{\sqrt{t}}\right),$$ where $u$ is given in Theorem \ref{thm:ak-pol} for this value of $\alpha$.  Then there exists $\tilde{t}_0\in (0,1)$, depending only on $\alpha$ and $\omega$, such that for all $\rho>0$ and $\theta\in \R$, there exists a solution $\Psi_1:(0,t_0]\times\R\to \mathbb{C}$ to the equation {\eqref{eqpsi-pol},}
%\begin{equation}\label{eqpsiresc-pol} i\partial_t \Psi+\partial_\sigma^2\Psi+\omega \frac{\Psi}{|\Psi|^2}=0, \end{equation}
which has the form
\begin{equation}
\Psi_1(t,\sigma)=\rho e^{i\theta + i\omega t}\left(r(t,\sigma)+\frac{H(t,\sigma)}{1+\psi(\alpha |\sigma|)}\right),
\end{equation}
where $\psi(\tau)=\tau \varphi(\tau)$, with $\varphi:\R_+\to [0,1]$ 
a smooth, cut-off function satisfying $\varphi=0$ on $[0,1]$ and $\varphi=1$ on $[2,+\infty)$, and where $r\in C([0,t_0], H^1(\R))$ satisfies  $\|r(t)\|_{L^\infty} < \sqrt{t}/4$ as $t\to 0$. Finally, we have
\begin{equation*}
 \left|\Psi_1(t,\sigma)\right|\geq
\begin{cases}\displaystyle  \frac{\rho}{4}(\sqrt{t}+\alpha |\sigma|)\quad \text{if}\quad \alpha |\sigma|\leq \frac 12,\\\,\\
\displaystyle \frac{\rho}{32}\quad \text{if}\quad \alpha |\sigma| \geq \frac 12.
\end{cases}
\end{equation*}

In particular, 
$$|\Psi_1(t,\sigma)|>0,\quad \forall (t,\sigma)\in (0,t_0]\times \R,\quad \Psi_1(0,0)=0,$$
and for all $t\in (0,t_0]$
$$\Psi_1(t,\sigma)\to z_1(t),\quad \text{as } \sigma \to \pm \infty.$$

\end{theorem}

%\begin{remark}The map $\Psi$ is a solution of \eqref{eqpsiresc} in the sense that it is a strong solution on $(0,t_0]\times \R^\ast$ and a mild solution on $(0,t_0]\times \R$.} \end{remark}}

\begin{corollary}
With the notations of the previous theorem, the filaments $(\Psi_j)$ with $\Psi_j=e^{i(j-1)\frac{2\pi}{N}  }\Psi_1, \forall \,1\leq j\leq N, \Psi_0=0$ are solutions to \eqref{KMD}, satisfying the parallelism condition at infinity \eqref{cond:parallel}, and colliding at time $t=0$ at position $\sigma=0$.
\end{corollary}

\begin{remark} In view of the estimates for $H$ established in {Subsection \S}\ref{sub:estim-H} the perturbation
$\Phi(t,\sigma)=r(t,\sigma)+{H(t,\sigma)}/({1+\psi(\alpha |\sigma|)})$ has finite (and possibly large) energy for all $t\in [0,t_0]$.
\end{remark}

The plan of this paper is the following. We devote next section to the proof of Theorem \ref{thm:main} by a fixed point argument. In Section \S\ref{sec:extension}, we sketch the proofs of Theorems \ref{thm:ak-pol} and \ref{thm:main-pol}, which are obtained as quite straightforward extensions of Theorem \ref{thm:ak} (that is, Theorem 3.1 in \cite{BFM}) and of Theorem \ref{thm:main} respectively. The last section gathers a collection of useful estimates that are needed for the fixed point argument in \S\ref{sec:main}.

\medskip

\textbf{Notation.} Throughout the paper, $C$ denotes a numerical constant that can possibly change from a line to another.

\section{Proof of Theorem \ref{thm:main}}\label{sec:main}
\subsection{The fixed point framework}

Let $K_0>K>1$, to be determined later, and let $\alpha>K_0$.
Let $t_0\in(0,1)$ to be determined later on in terms of $\alpha$. In all the following we denote by $I$ the interval $$I=\left(-\frac{1}{2\alpha},\frac{1}{2\alpha}\right).$$
We define the space
\begin{equation}\label{def:X-1}X=\left\{ r\in C([0,t_0], H^1(\R))\:|\quad \|r\|_{X}<\infty\right\},\end{equation}
where
\begin{equation}\label{def:X-2}\|r\|_{X} =\sup_{t\in (0,t_0]}
\left(\frac{\|r\|_{L^\infty((0,t),L^2)}}{ t^{3/4}}+\frac{\|\partial_\sigma r\|_{L^\infty((0,t),L^2)}}{ t^{\gamma}}+8\frac{\|r\|_{L^\infty((0,t)\times I)}}{t^{1/2}}\right).\end{equation}
The parameter $\gamma$ satisfies $0<\gamma<1/4$.

\bigskip

We look for a solution $$\Psi_1(t,\sigma)=-it + r(t,\sigma)+\frac{H(t,\sigma)}{1+\psi(\alpha |\sigma|)}$$ to the equation \eqref{eqpsi}. Hence we look for a solution $r$ to the equation
\begin{equation}\label{eq:r}
 i\partial_t r+\partial_\sigma^2 r=a(r)+b,
\end{equation}
where \begin{equation}\label{eq:source-1}
\begin{cases}
\displaystyle a(r)=\frac{1}{\re (r)+\frac{\re(H)}{1+\psi(\alpha |\sigma|)}}-1-\frac{1}{1+\psi(\alpha |\sigma|)}\frac{1}{\re (H)}\\ \\
\displaystyle b=-2 \ds H \partial_\sigma \left(\frac{1}{1+\psi(\alpha |\sigma|)}\right)-H\partial_\sigma^2\left(\frac{1}{1+\psi(\alpha |\sigma|)}\right).
\end{cases}
\end{equation} We choosed to gather in the expression of $a(r)$ also two source terms, in order to obtain a good control of $a(r)$, both near the origin and at infinity.

In order to solve \eqref{eq:r} we shall perform a fixed point argument in the subspace $B$ of $X$
\begin{equation}\label{def:B}B=\{r\in X\:| \quad \|r\|_X\leq 1\}\end{equation}
for the operator $A=A_a+A_b$, where
\begin{equation}
\begin{split}
&\label{op}A_a(r)(t,\sigma)=-i\int_0^t e^{i(t-s)\partial_\sigma^2}a(r)(s,\sigma)ds,\quad \mbox{and}\quad\\
& A_b(r)(t,\sigma)=-i\int_0^t e^{i(t-s)\partial_\sigma^2}b(s,\sigma)ds.
\end{split}
\end{equation}
To obtain the contraction for our fixed point argument, we shall also need to estimate
$$A(r_1)-A(r_2)=-i\int_0^t e^{i(t-s)\partial_\sigma^2}[a(r_1)-a(r_2)](s,\sigma)ds,$$
with 
$$a(r_1)-a(r_2)=\frac{\re(r_1-r_2)}{\left(\re (r_1)+\frac{\re(H)}{1+\psi(\alpha |\sigma|)}\right)\left(\re (r_2)+\frac{\re(H)}{1+\psi(\alpha |\sigma|)}\right)}.$$

Let us say a few words about the choice of the space $X$.
By definition of $B$, we have the pointwise estimate on $I$ 
\begin{equation}\label{ineq:upper-r-1} {\|r\|_{L^\infty([0,t]\times I)}<\frac{\sqrt{t}}{4},\quad \forall t\in [0,t_0],}\end{equation}
while outside $I$ we have by the Gagliardo-Nirenberg inequality
\begin{equation*}
\|r(t)\|_{L^\infty(\R\setminus I)}\leq \sqrt{2}\|r(t)\|_{L^2}\|\partial_\sigma r\|_{L^2}\leq \sqrt{2}t_0^{\frac{3}{4}+\gamma}.
\end{equation*}
Therefore if $t_0$ is smaller than an absolute constant we have 
\begin{equation}\label{ineq:upper-r}
 \|r\|_{L^\infty([0,t_0]\times\R)}<\frac{1}{32}.
\end{equation}
This will enable us (see \eqref{ineq:near-1}-\eqref{ineq:far-1}) to show that $$|r(t,\sigma)|\leq \frac{1}{2} \left(\frac{\re(H(t,\sigma))}{1+\psi(\alpha |\sigma|)}\right),$$ so $r$ will be negligible with respect to $\frac{\re(H(t,\sigma))}{1+\psi(\alpha |\sigma|)}$ at the denominator in \eqref{eq:source-1}. As a consequence we shall obtain in Lemma \ref{est} precise estimates, as for instance pointwise estimates on $a(r)$.\\

We start with a lemma containing the estimate that will allow us to perform the fixed point argument.

\begin{lemma}\label{fixedpointlemma}
We have the following control:
\begin{equation*}%\label{fixedpointineq}
\begin{split}
&\left\|\int_0^t e^{i(t-s)\partial_\sigma^2}F(s,\sigma)ds\right\|_X\\
&\leq C\sup_{t\in (0,t_0]}
\left(\frac{\|F\|_{L^1((0,t),L^2(\R))}}{t^\frac 34}
+\frac{\|\partial_\sigma F\|_{L^\frac 43 ((0,t),L^1(I))}+\|\partial_\sigma F\|_{L^1((0,t),L^2(\mathbb R\setminus I))}}{t^\gamma}\right)\\
&+C\sup_{t\in (0,t_0]}
\left(\frac{\int_0^t \|F(s)\|_{L^1(\frac{-1}\alpha,\frac 1\alpha)}(t-s)^{-\frac 12}\,ds}{t^\frac 12}\right)\\
&+C\sup_{t\in (0,t_0]}
\left(\frac{\int_0^t
(\alpha^\frac 32\|F(s)\|_{L^2(\R)}+\alpha^\frac 12\|\partial_\sigma F(s)\|_{L^2(\R)})(t-s)^\frac 12\,ds}{t^\frac12}\right).\end{split}
\end{equation*}
\end{lemma}

\begin{proof} We recall first a few basics facts on  the free Schr\"odinger evolution. 
The mass is conserved,
\begin{equation}\label{consmass}
\|e^{it\partial_\sigma^2}f\|_{L^2}=\|f\|_{L^2},
\end{equation}
and the dispersion inequality 
\begin{equation}\label{disp}
\|e^{it\partial_\sigma^2}f\|_{L^\infty}\leq\frac{C}{\sqrt{t}}\|f\|_{L^1},
\end{equation}
is valid, together with inhomogeneous Strichartz estimate (\cite{Ya}):
\begin{equation}\label{Str}
\left\|\int_0^t e^{i(t-s)\partial_\sigma^2}F(s)ds\right\|_{L^\infty((0,t_0),L^2)}\leq C\|F\|_{L^\frac 43((0,t_0),L^1)}.
\end{equation}

We denote by $A_F$ the Duhamel term to be estimated:
$$A_F(t)= \int_0^t e^{i(t-s)\partial_\sigma^2}F(s,\sigma)ds.$$
By the mass conservation \eqref{consmass} we obtain for $t\in(0,t_0]$
\begin{equation}\label{L2}
\begin{split}
\|A_F(t)\|_{L^2}
\leq C\|F\|_{L^1((0,t),L^2)}.\end{split}
\end{equation}

For the estimate of the gradient we use the conservation of mass \eqref{consmass} and the Strichartz estimate \eqref{Str}:
\begin{equation}\label{gradL2}
\begin{split}
\|\partial_\sigma A_F(t)\|_{L^2}&=\left\|\int_0^t e^{i(t-s)\partial_\sigma^2}\partial_\sigma F(s)\,ds\right\|_{L^2}
\\
&\leq C\left\|\partial_\sigma F\right\|_{L^\frac 43 ((0,t),L^1(I))}+C\left\|\partial_\sigma F\right\|_{L^1((0,t),L^2(\mathbb R\setminus I))}.
\end{split}
\end{equation}
Here we used $\partial_\sigma F = (\partial_\sigma F) \mathds{1}_I + (\partial_\sigma F) \mathds{1}_{\mathbb R\setminus I}$.

Finally we turn to the pointwise estimates near the origin, that are more delicate. We set $$\chi(\sigma)=(1-\varphi)(2\alpha |\sigma|),$$
where $\varphi$ is the cut-off function defined in Theorem \ref{thm:main}. In particular, $\chi$ is valued between $0$ and $1$, values $1$ on $I$ and vanishes outside $(-1/\alpha,1/\alpha)$. Since
\begin{equation*}
\left\|A_F(t)\right\|_{L^\infty(I)}
\leq 
\left\|\int_0^t \chi [e^{i(t-s)\partial_\sigma^2}F(s)]\,ds\right\|_{L^\infty},
\end{equation*}
we shall use the following commutator formula between a localization and a free Schr\"odinger evolution. 
For all $s,t\in \R$
\begin{equation*}
\chi e^{i(t-s)\partial_\sigma^2}f=e^{i(t-s)\partial_\sigma^2}[\chi f]
-i\int_s^t e^{i(t-\tau)\partial_\sigma^2}\left[\partial_\sigma^2\chi
 \,\,e^{i(\tau-s)\partial_\sigma^2}f+2\partial_\sigma\chi\,\,\partial_\sigma e^{i(\tau-s)\partial_\sigma^2}f\right]\,d\tau.
\end{equation*}
Using this with $s$ fixed and $f(\sigma)=F(s,\sigma),$ we get 
\begin{equation*}
\begin{split}
&\left\|A_F(t)\right\|_{L^\infty(I)}
\leq C\left\|\int_0^t e^{i(t-s)\partial_\sigma^2}[\chi F](s)\,ds\right\|_{L^\infty}\\
&+C\int_0^t \left\|\int _s^t e^{i(t-\tau)\partial_\sigma^2}\left[ \partial_\sigma^2\chi\,\,e^{i(\tau-s)\partial_\sigma^2}F(s)  
+2\partial_\sigma\chi\,\,\partial_\sigma (e^{i(\tau-s)\partial_\sigma^2}F(s))\right]\,d\tau\right\|_{L^\infty}\,ds.
\end{split}
\end{equation*}
Now we use the dispersion inequality \eqref{disp} to obtain
\begin{equation*}
\begin{split}
&\left\|A_F(t)\right\|_{L^\infty(I)}
\leq C\int_0^t \frac{1}{\sqrt{t-s}}\|F(s)\|_{L^1(-1/\alpha,1/\alpha)}\,ds\\
&+C\int_0^t\int_s^t \frac{1}{\sqrt{t-\tau}}
\left(\left\|\partial_\sigma^2\chi\,\,e^{i(\tau-s)\partial_\sigma^2}F(s) \right\|_{L^1}+\left\|\partial_\sigma\chi\,\,\partial_\sigma (e^{i(\tau-s)\partial_\sigma^2}F(s))\right\|_{L^1}\right)\,d\tau\,ds.
%\\
%&\leq C\int_0^t \frac{\|F(s)\|_{L^1(-1/\alpha,1/\alpha)}}{\sqrt{t-s}}\,ds+\int_0^t\sqrt{t-s}
%\left(\alpha^\frac 32\|F(s)\|_{L^2}+\alpha^\frac 12\|\partial_\sigma F(s)\|_{L^2}\right)\,ds.
\end{split}
\end{equation*}
Finally, by Cauchy-Schwarz inequality, we have
\begin{equation}\label{locLinfty}
\begin{split}
&\left\|A_F(t)\right\|_{L^\infty(I)}
\leq C\int_0^t \frac{1}{\sqrt{t-s}}\|F(s)\|_{L^1(-1/\alpha,1/\alpha)}\,ds\\
&+C\int_0^t\sqrt{t-s}
\left(\alpha^\frac 32\|F(s)\|_{L^2}+\alpha^\frac 12\|\partial_\sigma F(s)\|_{L^2}\right)\,ds\end{split}
\end{equation}
as 
$$
\|\partial_\sigma^p  \chi\|_{L^2} \leq C \alpha^{p - 1/2} \|\partial_\sigma^p\varphi\|_{L^2}. 
$$
Gathering \eqref{L2}, \eqref{gradL2} and \eqref{locLinfty} we obtain the lemma. 
\end{proof}

Next we gather in the following lemma several estimates that will allow us to end the fixed point argument. 
\begin{lemma}\label{est} Let $r\in B$ and $0\leq s\leq t_0$.

\textbullet\: The following pointwise estimates hold:
\begin{equation}\label{ineq:a-1}
|a(r)(s,\sigma)|\leq 1+C\frac{|r(s,\sigma)|}{(\sqrt{s}+\alpha |\sigma|)^2}\quad \mbox{ if } \sigma\in I,
\end{equation}
\begin{equation}\label{ineq:a-2}
|a(r)(s,\sigma)|\leq C\left(\frac{\alpha}{1+\psi(\alpha |\sigma|)}+|r(s,\sigma)|\right)\quad \mbox{ if } \sigma\in \R\setminus I.
\end{equation}

\textbullet\: For $a(r)$ the following norm estimates hold:
\begin{equation}\label{aL2}
\|a(r)(s)\|_{L^2(\mathbb R)}
\leq C(\alpha^{\frac 12}+\alpha^{-\frac 12} s^{-\frac 14}),
\end{equation}
\begin{equation}\label{aL1}
\|a(r)(s)\|_{L^1(\frac{-1}\alpha,\frac 1\alpha)}\leq C \alpha^{-\frac 12},
\end{equation}
\begin{equation}\label{gradaL2}
\|\partial_\sigma a(r)(s)\|_{L^2(\mathbb R\setminus I)}\leq C(s^\gamma+\alpha^\frac 12),
\end{equation}
\begin{equation}\label{gradaL2I}
\|\partial_\sigma a(r)(s)\|_{L^2(I)}\leq C(s^{\gamma-1}+\alpha^\frac 12 s^{-\frac 34}),
\end{equation}
\begin{equation}\label{gradaL1}
\|\partial_\sigma a(r)(s)\|_{L^1(I)}\leq C(\alpha^{-\frac12} s^{\gamma-\frac 34}+s^{-\frac 12}).
\end{equation}

\textbullet\: For $b$ the following estimates hold:
\begin{equation}\label{bL2}
\|b(s)\|_{L^2(\mathbb R)}\leq C\alpha^\frac 32,
\end{equation}
\begin{equation}\label{bL1}
\|b(s)\|_{L^1(\frac{-1}\alpha,\frac 1\alpha)}=0,
\end{equation}
\begin{equation}\label{gradbL2}
\|\partial_\sigma b(s)\|_{L^2(\mathbb R\setminus I)}\leq C(\alpha^\frac 52+\alpha^\frac 32 s^{-\frac 12}),
\end{equation}
\begin{equation}\label{gradbL2I}
\|\partial_\sigma b(s)\|_{L^2(I)}=0,
\end{equation}
\begin{equation}\label{gradbL1}
\|\partial_\sigma b(s)\|_{L^1(I)}=0,
\end{equation}

\textbullet\: For $r_1, r_2\in B$ the following estimates hold:
\begin{equation}\label{cL2}
\|a(r_1)(s)-a(r_2)(s)\|_{L^2(\mathbb R)}\leq C\|r_1-r_2\|_{X}(\alpha^{-\frac12}  s^{-\frac14} +s^{\frac 34}),
\end{equation}
\begin{equation}\label{cL1}
\|a(r_1)(s)-a(r_2)(s)\|_{L^1(\frac{-1}\alpha,\frac 1\alpha)}\leq C\|r_1-r_2\|_{X}\alpha^{-\frac 12},
\end{equation}
\begin{equation}\label{gradcL2}
\|\partial_\sigma [a(r_1)-a(r_2)](s)\|_{L^2(\mathbb R\setminus I)}\leq C\|r_1-r_2\|_{X}(s^\gamma+\alpha^\frac 12s^{\frac 38+\frac\gamma 2}),
\end{equation}
\begin{equation}\label{gradcL2I}
\|\partial_\sigma [a(r_1)-a(r_2)](s)\|_{L^2(I)}\leq C\|r_1-r_2\|_X (s^{\gamma-1}+\alpha s^{-\frac 34}),
\end{equation}
\begin{equation}\label{gradcL1}
\|\partial_\sigma [a(r_1)-a(r_2)](s)\|_{L^1(I)}\leq C\|r_1-r_2\|_X \,( \alpha^{-\frac12}s^{\gamma-\frac 34} +s^{-\frac12}).
\end{equation}

\end{lemma}

We postpone the technical proof of this lemma to the last section.

\subsection{Proof of Theorem \ref{thm:main} completed}

We can now show that the operators $A_a$ and $A_b$ are contractions on the subset $B$ of $X$.

\medskip

\textbullet \: \textbf{Stability.}
Let $r\in B$.
We apply Lemma \ref{fixedpointlemma} to $F=a(r)$ and we use \eqref{aL2}, \eqref{aL1}, \eqref{gradaL2}, \eqref{gradaL2I} and \eqref{gradaL1} to obtain 
\begin{equation}\label{estim:A}\begin{split}\|A_a(r)\|_X&
\leq C\sup_{t\in (0,t_0]}
\left(\frac{\alpha^{\frac{1}{2}}t+\alpha^{-\frac 12}t^\frac 34}{t^\frac 34}+\frac{\alpha^{-\frac 12}t^{\gamma}+t^{\frac 14}+t^{\gamma+1}+\alpha^\frac 12 t}{t^\gamma}\right)\\
&+C\sup_{t\in (0,t_0]}\left(\frac{\alpha^{-\frac 12} t^{\frac 12}+\alpha^2t^{\frac 32}+\alpha t^{\frac 54}+\alpha^\frac 12 t^{\gamma+\frac 12}+\alpha t^\frac 34}{t^\frac 12}\right).\end{split}\end{equation}

In order to estimate $A_b$ in $X$ we apply Lemma \ref{fixedpointlemma} to $F=b$. Then the estimates  \eqref{bL2}, \eqref{bL1}, \eqref{gradbL2}, \eqref{gradbL2I} and \eqref{gradbL1} for $b$ yield
\begin{equation*}\label{est:A}\|A_b(r)\|_X\leq C\sup_{t\in (0,t_0]}\left(\frac{\alpha^\frac 32 t}{t^\frac 34}+
 \frac{\alpha^{\frac52} t + \alpha^{\frac32} t^\frac12}{t^\gamma}
%C\frac{\alpha^\frac 32 t}{t^\gamma}
+\frac{\alpha^3t^\frac 32+\alpha^2 t}{t^\frac 12}\right).\end{equation*}

Hence we obtain an estimate of the form
$$
\left\|A(r)\right\|_{X}\leq C\left(\alpha^{-\frac 12}+\sum_{k=1}^{13}\alpha^{r_k}t_0^{p_k}\right),
$$where $r_k\geq 0$ and $p_k>0$. We choose first $K_0$ such that $C\alpha^{-\frac 12}\leq 1/14.$ 	Then we choose $t_0$ sufficiently small such that $C\sum_k\alpha^{r_k}t_0^{p_k}\leq 13/14$.
We conclude that $\left\|A(r)\right\|_{X}\leq 1$.

\medskip

\textbullet \:\textbf{Contraction.} 
Again, in order to control $A(r_1)-A(r_2)$ we use Lemma \ref{fixedpointlemma} with the choice $F=a(r_1)-a(r_2)$. Therefore, combining with the estimates \eqref{cL2}, \eqref{cL1}, \eqref{gradcL2}, \eqref{gradcL2I} and \eqref{gradcL1} we obtain
\begin{equation*}\begin{split}
&\|A(r_1)-A(r_2)\|_X \\
&\leq C\|r_1-r_2\|_X\sup_{t\in(0,t_0]}\left(\frac{\alpha^{-\frac 12}t^{\frac34}+t^\frac 74}{t^\frac 34}+\frac{\alpha^{-\frac 12} t^\gamma+ t^{\frac14} +t^{\gamma+1}+ \alpha^\frac 12t^{\frac {11}8+\frac \gamma 2}}{t^\gamma}\right)\\
&+ C\|r_1-r_2\|_X\sup_{t\in(0,t_0]}\left(\frac{\alpha^{-\frac 12}t^{\frac 12}+\alpha t^\frac 54 +\alpha^\frac 32t^{\frac 94}+\alpha^\frac 12 t^{\gamma+\frac 12}+\alpha^\frac 32 t^\frac 34+\alpha t^{\frac{15}{8} + \frac{\gamma}{2}}}{t^\frac 12}\right)\\
&\leq C\left(\alpha^{-\frac 12}+\sum_{k=1}^{9}\alpha^{r'_k}t_0^{p'_k}\right)\|r_1-r_2\|_X\end{split}\end{equation*}where $r'_k\geq 0$ and $p'_k>0$.
Therefore we obtain
$$\|A(r_1)-A(r_2)\|_X\leq \frac{1}{2}\|r_1-r_2\|_X$$
 provided that again $\alpha>K_0$ with $K_0$ sufficiently large, and  $t_0$ is sufficiently small with respect to $\alpha$.

\section{Proof of Theorems \ref{thm:ak-pol} and \ref{thm:main-pol}}

\label{sec:extension}

\subsection{Sketch of proof of Theorem \ref{thm:ak-pol}}
In this paragraph, we briefly sketch the arguments for the proof of Theorem \ref{thm:ak-pol}, which follows the lines of the proof of Theorem \ref{thm:ak}, given in \cite[Theo. 3.1]{BFM}.

Let us recall that the strategy to obtain a unique solution $(u,v)$ to \eqref{syst:self-sim-pol}, as given in the statement of Theorem \ref{thm:ak-pol}, relies on finding a fixed point $w=v-1$ in the space $E$ for the operator 
$$P(w)(x)= e^{\frac{ix^2}{4}}-1+e^{\frac{ix^2}{4}}\omega \int_0^x \frac{y e^{-\frac{iy^2}{4}}}{\overline{u(w)}(y)}\,dy,\quad x\in \R,$$
with $u(w)$ dictated by the second equation of \eqref{syst:self-sim-pol} together with the condition at infinity \eqref{infcond-pol}:
$$
u(w)(x)%&=&1+ \alpha|x|  +|x| \int_{|x|}^{+\infty} \frac{w(z)}{z^2} dz\\
=1+|x|\left(\alpha+\int_{|x|}^{+\infty} \frac{w(z)}{z^2}\,dz\right),\quad \text{for  } x\in \R.
$$
The difference between Theorems \ref{thm:ak-pol} and \ref{thm:ak} is that the nonlinearity is given by $1/\overline{u}=u/|u|^2$ instead of $1/\re(u)$. It turns out that the estimates of the proof of \cite[Theo.\ 3.1]{BFM} may be adapted in a straightforward way, essentially because $|u|\geq \re(u)$. More precisely, given the inequalities recalled in  Lemma \ref{lemma:borrowed} we have on the one hand for any $w\in E$ $$|u(w)(x)|\geq \re(u(w)(x))\geq 1+\frac{\re(\alpha) |x|}{2},$$ and on the other hand
$$  |u(w)'(x)|\leq 2|\alpha|,$$ whence
$$\left|\frac{\overline{u(w)}'(x)}{\overline{u(w)}(x)^2}\right| \leq \frac{C|\alpha|}{\re(\alpha)^2}.$$ This imposes the additional condition $|\alpha|\leq \re(\alpha)^2$ in order to close the fixed point argument.

\subsection{Proof of Theorem \ref{thm:main-pol}}
We look for a solution $$\Psi_1(t,\sigma)=\rho e^{i\theta+i\omega t}\left(r(t,\sigma)+\frac{H(t,\sigma)}{(1+\psi(\alpha |\sigma|)}\right)$$  to the equation \eqref{eqpsi-pol}. Hence we look for a solution $r$ to the equation
\begin{equation*}
 i\partial_t r+\partial_\sigma^2 r=\omega \tilde{a}(r)+b,
\end{equation*}
where \begin{equation*}
\begin{cases}
\displaystyle \tilde{a}(r)=- \frac{1}{\overline{r}+\frac{\overline{H}}{1+\psi(\alpha |\sigma|)}}
+{r}+\frac{{H}}{1+\psi(\alpha |\sigma|)}+\frac{1}{(1+\psi(\alpha |\sigma|))\overline{H}}\\ \\
\displaystyle b=-2 \ds H \partial_\sigma \left(\frac{1}{1+\psi(\alpha |\sigma|)}\right)-H\partial_\sigma^2\left(\frac{1}{1+\psi(\alpha |\sigma|)}\right).
\end{cases}
\end{equation*} 
%Let  $r_1$ and $r_2\in X$. We set
%$$\tilde{c}= \tilde{a}(r_2)-\tilde{a}(r_1)=
%\frac{\omega(\overline{r_1}-\overline{r_2})}{(\overline{r_1}+\frac{\overline{H}}{1+\psi(\alpha |\sigma|)})(\overline{r_2}+\frac{\overline{H}}{1+\psi(\alpha |\sigma|)})}+\omega(r_2-r_1).$$
Note that we get the same source term $b$ as in \eqref{eq:source-1}. In particular, in order to transpose the proof of Theorem \ref{thm:main}, we only have to establish the analogs of estimates \eqref{ineq:a-1} to \eqref{gradaL1} for $\tilde{a}$, and the estimates \eqref{cL2} to \eqref{gradcL1} for $\tilde{a}(r_2)-\tilde{a}(r_1)$. 

More precisely, let $\tilde{t}_0\in (0,1)$ to be determined later and let $X$ and $B$ denote the corresponding spaces defined as in \eqref{def:X-1}, \eqref{def:X-2} and \eqref{def:B}. 
\begin{lemma}\label{est-pol} Let $r\in B$ and $0\leq s\leq \tilde{t}_0$.

\textbullet\: For $\tilde{a}(r)$,  the following pointwise estimates hold:
\begin{equation}\label{ineq:a-1-pol}
|\tilde{a}(r)(s,\sigma)|\leq C\left(1+\frac{|r(s,\sigma)|}{(\sqrt{s}+\alpha |\sigma|)^2}\right)\quad \mbox{ if } \sigma\in I,
\end{equation}
\begin{equation}\label{ineq:a-2-pol}
|\tilde{a}(r)(s,\sigma)|\leq C\left( \frac{\alpha^2}{(1+\psi(\alpha |\sigma|))^2}+\frac{\alpha}{1+\psi(\alpha |\sigma|)}+|r(s,\sigma)|\right)\quad \mbox{ if } \sigma\in \R\setminus I.
\end{equation}

\textbullet\: The following norm estimates hold:
\begin{equation}\label{aL2-pol}
\|\tilde{a}(r)(s)\|_{L^2(\mathbb R)}\leq C(\alpha^{\frac 32}+\alpha^{-\frac 12}s^{-\frac 14}).
\end{equation}
Moreover, the estimates \eqref{aL1} to \eqref{gradaL1} hold true for $\tilde{a}(r)$.

\medskip

\textbullet\: Let $r_1\in B$, $r_2\in B$. The estimates \eqref{cL2} to \eqref{gradcL1} hold true for $\tilde{a}(r_1)-\tilde{a}(r_2)$.

\medskip

Therefore, we have the same estimates for $\tilde{a}(r)$ and $\tilde{a}(r_1)-\tilde{a}(r_2)$ except for \eqref{ineq:a-2-pol} and for the exponent of $\alpha$ in the first term in the right-hand side of  \eqref{aL2-pol}.
\end{lemma}
The technical proof of Lemma \ref{est-pol} is provided in the last section.

\medskip

We can now complete the proof of Theorem \ref{thm:main-pol}. In view of Lemma \ref{est-pol}, the only difference with the proof of Theorem \ref{thm:main} appears when estimating $\|A_{\tilde{a}}(r)\|_X$, where
$$A_{\tilde{a}}(r)(t,\sigma)=-i\omega \int_0^t e^{i(t-s)\partial_\sigma^2}\tilde{a}(r)(s,\sigma)ds.$$
More precisely, there are only two terms involving the norm $\|\tilde{a}(r)\|_{L^2}$ in Lemma \ref{fixedpointlemma} with $F=\tilde{a}(r)$. Hence only two terms are changed in the estimate \eqref{estim:A}, which turns into:
\begin{equation*}\begin{split}\|A_{\tilde{a}}(r)\|_X&
\leq C|\omega|\sup_{t\in (0,\tilde{t}_0]}
\left(\frac{\alpha^{\frac{3}{2}}t+\alpha^{-\frac 12}t^\frac 34}{t^\frac 34}+\frac{\alpha^{-\frac 12}t^{\gamma}+t^{\frac 14}+t^{\gamma+1}+\alpha^\frac 12 t}{t^\gamma}\right)\\
&+C|\omega|\sup_{t\in (0,\tilde{t}_0]}\left(\frac{\alpha^{-\frac 12} t^{\frac 12}+\alpha^3 t^{\frac 32}+\alpha t^{\frac 54}+\alpha^\frac 12 t^{\gamma+\frac 12}+\alpha t^\frac 34}{t^\frac 12}\right).\end{split}\end{equation*}
With the same arguments as in the proof of Theorem \ref{thm:main}, we conclude that
$$
\left\|A_{\tilde{a}}(r)+A_b(r)\right\|_{X}< 1
$$
and stability holds provided $\alpha>\tilde{K}_0$ with $\tilde{K}_0$ sufficiently large, and  $\tilde{t}_0$ is sufficiently small with respect to $\alpha$.

\section{Some useful estimates and the proofs of Lemmas \ref{est} and \ref{est-pol}}

We first recall and collect a few estimates for the profile $H$, which are borrowed from \cite{BFM}. 
\subsection{Estimates on $H$}\label{sub:estim-H}
The following Lemma is directly derived from the estimates (19) to (21) in \cite{BFM}.
\begin{lemma}\label{lemma:borrowed}
Let $\alpha\in  \C$ such that $\re(\alpha)>0$. Let $$E=\left\{w\in C^1(\R),\:w(0)=w'(0)=0,\:\: w\text{  even  },\:\: \|w\|_{L^\infty} + \||x|^{-1}w'\|_{L^\infty}\leq \frac{\re(\alpha)}{4}\right\}.$$
Let $w\in E$. We have
\begin{equation}\label{ineq:app2}
\sup_{x\in \R_+}\frac{|w(x)|}{x}\leq \|w\|_{L^\infty}+ \frac{1}{2}\| |x|^{-1} w'\|_{\infty}
\end{equation}
and
\begin{equation}\label{ineq:app4}
\int_{0}^{+\infty}\frac{|w|}{z^2}\,dz\leq \|w\|_{L^\infty}+ \frac{1}{2}\| |x|^{-1} w'\|_{\infty}.
\end{equation}

Defining
$$u(x)=1+\alpha |x|+|x|\int_{|x|}^{+\infty}\frac{w}{z^2}\,dz,\quad x\in \R,$$
the following estimates hold:
\begin{equation*}
|u(x)-1-\alpha |x||\leq \min\left(\|w\|_{L^\infty}, |x|\int_0^{+\infty} \frac{|w|}{z^2}\,dz\right)\leq \min\left( 1,|x|\right)\frac{\re(\alpha)}{4},
\end{equation*}
\begin{equation*}
|u(x)|\leq 1+|x|\left(|\alpha|+\int_{0}^{+\infty}\frac{|w|}{z^2}\,dz\right),
\end{equation*}
and
therefore
\begin{equation*}\label{ineq:u}
|u(x)|\leq 1+  \frac{5|\alpha||x|}{4}. 
\end{equation*}
Moreover,
\begin{equation}\label{ineq:lower-appendix}
|u(x)|\geq \re (u(x))\geq 1+\frac{3\re(\alpha) |x|}{4}\geq \frac{1}{2}(1+\re(\alpha) |x|).
\end{equation}

\medskip

Finally, we have $u\in C^2(\R^*)$ and for all $x\in \R^*$
 \begin{equation}\label{ineq:app3}
| u'(x)| \leq |\alpha| +\left|\int_{|x|}^{+\infty} \frac{w}{z^2} d z\right| + \frac{|w(x)|}{|x|} \leq 2|\alpha|, 
\end{equation}
and 
 \begin{equation}\label{ineq:app5}
|u''(x)| =  \frac{|w'(x)|}{|x|} \leq \frac{ |\alpha|}{4}.
\end{equation}

\end{lemma}

We infer immediately the following estimates for $H$
\begin{lemma}\label{lemma:H}Let $\alpha>0$ and $u$ as in Lemma \ref{lemma:borrowed}. Let $s\in (0,1)$.
Setting $$H(s,\sigma)=\sqrt{s}u\left(\frac{\sigma}{\sqrt{s}}\right),\quad \sigma\in \R,$$ we have
\begin{equation}\label{dev:H}
|H(s,\sigma)-\sqrt{s}-\alpha |\sigma||\leq \frac{\alpha}{4}
\min(\sqrt{s},|\sigma |)
\end{equation}
and in particular 
\begin{equation}\label{ineq:upper-bound-2}
|H(s,\sigma)|\leq C(\sqrt{s}+\alpha |\sigma|).
\end{equation}
Moreover,
\begin{equation}\label{Hinffar}
|H(s,\sigma)|\geq \re (H(s,\sigma))\geq \frac{1}{2}(\sqrt{s}+\alpha |\sigma|). 
\end{equation}

Finally,
$$
| \partial_\sigma H(s,\sigma) | \leq 2 \alpha \quad \mbox{and}\quad
| \partial_\sigma^2 H(s,\sigma) | \leq \frac{\alpha}{4}  s^{-1/2},\quad \text{for s } \neq 0 \quad \text{and}\quad \sigma\neq 0.
$$

\end{lemma}

\begin{lemma}\label{lemma:H-r}
Under the same assumptions as in Lemmata \ref{lemma:borrowed} and \ref{lemma:H} we have, setting $I=(-1/2\alpha,1/2\alpha)$
\begin{equation}\label{ineq:near-1}
\frac{| H(s,\sigma)|}{1+\psi(\alpha |\sigma|)}\geq \frac{\re (H(s,\sigma))}{1+\psi(\alpha |\sigma|)}\geq \frac{1}{2}(\sqrt{s}+\alpha |\sigma|)\quad \text{if }\: \sigma \in I,
\end{equation}
and
\begin{equation}\label{ineq:far-1}
\frac{|H(s,\sigma)|}{1+\psi(\alpha |\sigma|)}\geq \frac{\re (H(s,\sigma))}{1+\psi(\alpha |\sigma|)}\geq \frac{1}{3}\quad \text{if }\: \sigma \in \R\setminus I.
\end{equation}

Let $r\in X$, where $X$ is defined in \eqref{def:X-1}-\eqref{def:X-2}, and assume that $\|r\|_X\leq 1$. We have
\begin{equation}\label{ineq:near-2}
\left|r(s,\sigma))+\frac{H(s,\sigma)}{1+\psi(\alpha |\sigma|)}\right|\geq \re (r(s,\sigma))+\frac{\re(H(s,\sigma))}{1+\psi(\alpha |\sigma|)}\geq \frac{1}{4}(\sqrt{s}+\alpha |\sigma|)\quad \text{if }\: \sigma\in I,
\end{equation}
and 
\begin{equation}\label{ineq:far-2}
\left|r(s,\sigma))+\frac{H(s,\sigma)}{1+\psi(\alpha |\sigma|)}\right|\geq \re (r(s,\sigma))+\frac{\re(H(s,\sigma))}{1+\psi(\alpha |\sigma|)}\geq \frac{1}{32}\quad \text{if }\: \sigma \in \R\setminus I.
\end{equation}

\end{lemma}

\begin{proof}[Proof of Lemma \ref{lemma:H-r}]
We immediately infer \eqref{ineq:near-1} from \eqref{Hinffar} since $\psi(\alpha |\sigma|)=0$ for $\sigma\in I$. 
Moreover, we have using \eqref{Hinffar} and the fact that $\psi(\alpha |\sigma|)\leq \alpha |\sigma|$
$$
\frac{\re (H(s,\sigma))}{1+\psi(\alpha |\sigma|)}\geq \frac 12 \left(\frac{\sqrt{s}+\alpha |\sigma|}{1+\psi(\alpha |\sigma|)}\right) \geq \frac 12 \left(\frac{\alpha |\sigma|}{1+\alpha |\sigma|}\right). 
$$
Since $x/(1+x)\geq 2/3$ for any $x\geq 2$ we obtain \eqref{ineq:far-1}.
Next, by \eqref{ineq:upper-r-1} we have for $\sigma \in I$
\begin{equation*}\begin{split}
 \re (r(s,\sigma))+\frac{\re(H(s,\sigma))}{1+\psi(\alpha |\sigma|)}
&\geq \frac{1}{2}(\sqrt{s}+\alpha |\sigma|)-\|r(s)\|_{L^\infty(I)}\geq \frac{1}{4}(\sqrt{s}+\alpha |\sigma|)\end{split}
\end{equation*}
and \eqref{ineq:near-2} follows.

Finally, using \eqref{ineq:upper-r}, we obtain for $\sigma\in \R\setminus I$
\begin{equation*}
\begin{split}
 \re (r(s,\sigma))+\frac{\re(H(s,\sigma))}{1+\psi(\alpha |\sigma|)}\geq
\frac{1}{3}-\frac{1}{32}\geq 
 \frac{1}{32},
\end{split}
\end{equation*} which yields \eqref{ineq:far-2}.

\end{proof}

\subsection{Estimates on $a(r)$}
\medskip
Now we establish the pointwise estimates for 
$$ 
a(r)=\frac{1}{\re (r)+\frac{\re(H)}{1+\psi(\alpha |\sigma|)}}-1-\frac{1}{1+\psi(\alpha |\sigma|)}\frac{1}{\re (H)}.$$
For simplicity we shall write $a$ instead of $a(r)$.
For $\sigma \in I$ we have $\psi(\alpha |\sigma|)=0$, hence 
$$a=\frac{1}{\re (r+H)}-\frac{1}{\re (H)}-1=\frac{-\re (r)}{\re (r+H)\re (H)}-1,$$
so \eqref{ineq:a-1} follows from \eqref{ineq:near-1} and \eqref{ineq:near-2}.

For $\sigma \in \R\setminus I$ we have using \eqref{ineq:far-1} and \eqref{ineq:far-2}
\begin{equation}\label{ineq:a-3}
\begin{split}
|a(s,\sigma)|&\leq \frac{1}{1+\psi(\alpha |\sigma|)}\frac{1}{\re (H)}
+\left|\frac{1}{\re \left(r+\frac{H}{1+\psi(\alpha |\sigma|)}\right)}-1\right|\\
&\leq \frac{C}{(1+\psi(\alpha |\sigma|))^2}+\frac{| 1+\psi(\alpha |\sigma|)-\re (H)|}
{(1+\psi(\alpha |\sigma|))\re \left(r+\frac{ H}{1+\psi(\alpha |\sigma|)}\right)}
+\frac{|\re (r(s,\sigma))|}{\re \left(r+\frac{ H}{1+\psi(\alpha |\sigma|)}\right)}.
\end{split}
\end{equation}
Next, we have by \eqref{dev:H} 
$$|1+\psi(\alpha |\sigma|)-\re(H)|\leq |1+\alpha |\sigma|-H| + |\psi(\alpha |\sigma|) - \alpha |\sigma||  \leq  C(\alpha \sqrt{s} + 1) \leq C \alpha,$$
as $\alpha\geq 1$.  
Hence by \eqref{ineq:far-2} and \eqref{ineq:a-3} we obtain
\begin{equation*}\label{ineq:a-3}
\begin{split}
|a(s,\sigma)|
&\leq \frac{C}{(1+\psi(\alpha |\sigma|))^2}+\frac{C\alpha }{1+\psi(\alpha |\sigma|)}
+C{|\re (r(s,\sigma))|}\\
&\leq \frac{C\alpha }{1+\psi(\alpha |\sigma|)}
+C{|\re (r(s,\sigma))|},
\end{split}
\end{equation*}
which is \eqref{ineq:a-2}. 

We use now the pointwise estimates  \eqref{ineq:a-1} and \eqref{ineq:a-2} to get $L^2$ estimates:
\begin{equation*}
\begin{split}
\|a(s)\|_{L^2}&\leq \|a(s)\|_{L^2(I)}+\|a(s)\|_{L^2(\mathbb R\setminus I)}\\
&\leq \|1\|_{L^2(I)}+\|r(s)\|_{L^\infty(I)} \left\| \frac{1}{ (\sqrt{s}+\alpha |\sigma|)^2}\right\|_{L^2(I)}\\
&
+C\left\| \frac{\alpha}{(1+\psi(\alpha |\sigma|))}\right\|_{L^2(\mathbb R\setminus I)}+\|r(s)\|_{L^2(\mathbb R\setminus I)}.\end{split}\end{equation*}
Since
\begin{equation}\label{estint1}
\left\|\frac{1}{(\sqrt{s}+\alpha |\sigma|)^p}\right\|_{L^2}=C\alpha^{-\frac{1}{2}}s^{\frac{1}{4}-\frac{p}2}, \quad p \in \mathbb{N}^*, 
\end{equation}
\begin{equation}\label{schumann}
\left\|\frac 1{(\sqrt{s}+\alpha|\sigma|)^p}\right\|_{L^1} = C \alpha^{-1}s^{-\frac{p}2 +\frac12}, \quad p >1,
\end{equation}
and
\begin{equation}\label{estint2}
\left\| \frac{1}{(1+\psi(\alpha |\sigma|))^p}\right\|_{L^2}=C\alpha^{-\frac{1}{2}},\quad p \in \mathbb{N}^*. 
\end{equation}
we obtain
$$\|a(s)\|_{L^2}\leq C(\alpha^{-\frac 12}+ \alpha^{-\frac 12}s^{-\frac 14}+\alpha^{\frac 12}+s^\frac 34),$$
and \eqref{aL2} follows.

We now estimate the $L^1$ norm near zero.
We have from the pointwise estimate \eqref{ineq:a-1} and the Cauchy-Schwarz inequality
\begin{equation}\label{aL1I}
\|a(s)\|_{L^1(I)}\leq \|1\|_{L^1(I)}+ C\|r(s)\|_{L^2}\left\|  \frac{1}{(\sqrt{s}+\alpha|\sigma|)^2}   \right\|_{L^2(I)}\leq \alpha^{-1}+C\alpha^{-\frac{1}{2}}.
\end{equation}
Next, for $1/2\leq \alpha|\sigma| \leq 1$ we have $\psi(\alpha |\sigma|)=0$, hence \eqref{ineq:a-3} yields
\begin{equation*}
\begin{split}
|a(s,\sigma)|
&
\leq \frac{1}{\re(H)} + \frac{1 + |\re (r)|}{\re \left(r+ H\right)}
+ \frac{| \re (H)|}
{\re \left(r+\ H\right)}
\end{split}
\end{equation*}
therefore using \eqref{ineq:upper-r}, \eqref{ineq:far-1}, \eqref{ineq:far-2} and \eqref{ineq:upper-bound-2} we get
\begin{equation*}
|a(s,\sigma)|\leq C.
\end{equation*}
It follows that 
\begin{equation*}
\|a(s)\|_{L^1(1/2\leq \alpha |\sigma|\leq 1)}\leq C\alpha^{-1},
\end{equation*}
so together with \eqref{aL1I} we obtain \eqref{aL1}.

\medskip

We finally turn to the estimate of the gradient.  We compute
\begin{equation*}\begin{split}
\partial_\sigma a(s,\sigma)&=-\frac{\partial_\sigma \re(r)+\partial_\sigma \left(\frac{\re(H)}{1+\psi(\alpha|\sigma|)}\right)}
{\left( \re(r)+\frac{\re(H)}{1+\psi(\alpha|\sigma|)}\right)^2}+\frac{1}{\re(H)}\frac{\alpha \text{sgn}(\sigma)\psi'(\alpha |\sigma|)}{(1+\psi(\alpha |\sigma|))^2}\\&+\frac{\partial_\sigma \re(H)}{(\re(H))^2(1+\psi(\alpha |\sigma|))}.\end{split}\end{equation*}

For $\sigma\in\mathbb R\setminus I$ we use \eqref{ineq:far-1}, \eqref{ineq:far-2}, and the fact that $|\partial_\sigma H(s,\sigma)|\leq 2\alpha$ to obtain
\begin{equation}\label{gradaLinfty}
|\partial_\sigma a(s,\sigma)|\leq C|\partial_\sigma r|+C
\left|\partial_\sigma \left(\frac{\re(H)}{1+\psi(\alpha|\sigma|)}\right)\right|
+\frac{C\alpha}{1+\psi(\alpha|\sigma|)}.
\end{equation}
Now, we claim that for all $\sigma\in \R^*$, 
\begin{equation}
\label{chili}
\left|\partial_\sigma \left(\frac{H}{1+\psi(\alpha|\sigma|)}\right)\right| \leq  \frac{C\alpha}{1+\psi(\alpha|\sigma|)}. 
\end{equation}
Indeed, since $|\partial_\sigma H(s,\sigma)|\leq 2\alpha$, we have by \eqref{ineq:upper-bound-2}
\begin{equation}\begin{split}
\left|\partial_\sigma \left(\frac{H}{1+\psi(\alpha|\sigma|)}\right)\right|
&\leq  \frac{C\alpha}{1+\psi(\alpha|\sigma|)} + \frac{\alpha |H|}{(1+\psi(\alpha|\sigma|))^2}\\
&\leq \frac{C\alpha}{1+\psi(\alpha|\sigma|)}+\frac{\alpha \sqrt{s} + \alpha^2 |\sigma|}{(1+\psi(\alpha|\sigma|))^2} \\
&\leq \frac{C \alpha}{1+\psi(\alpha|\sigma|)} \left(1+\frac{\sqrt{s} + \alpha |\sigma|}{1+\psi(\alpha|\sigma|)}\right)
\end{split}\end{equation}
%\begin{equation}
%\label{robert}
%\frac{\alpha \sqrt{s} + \alpha^2 |\sigma|}{(1+\psi(\alpha|\sigma|))^2} \leq \frac{\alpha}{1+\psi(\alpha|\sigma|)} \left(\frac{\sqrt{s} + \alpha |\sigma|}{1+\psi(\alpha|\sigma|)}\right)
%\end{equation}
and this gives the results for for $s < 1$, as the function $x \mapsto (1 + x)/(1 + \psi(x))$ is bounded on $\R_+$.

Using \eqref{estint2} and the estimate  $\|\partial_\sigma r(t,\sigma)\|_{L^2} \leq t^\gamma$, we obtain \eqref{gradaL2}.

\medskip

For $\sigma\in I$ we have $\psi(\alpha|\sigma|)=0$ so in view of \eqref{ineq:near-1} and \eqref{ineq:near-2} we get
\begin{equation}\label{gradaLinftybis}
|\partial_\sigma a(s,\sigma)|\leq C\frac{|\partial_\sigma r|+\alpha}{(\sqrt{s}+\alpha|\sigma|)^2}\leq  C\left(\frac{|\partial_\sigma r|}{s}+\frac\alpha {(\sqrt{s}+\alpha|\sigma|)^2}\right).
\end{equation}
Therefore, using \eqref{estint1} we obtain \eqref{gradaL2I}. Estimate \eqref{gradaL1} follows from \eqref{estint1} and \eqref{schumann} by using Cauchy-Schwarz inequality:
$$\|\partial_\sigma a(s)\|_{L^1(I)}\leq C\|\partial_\sigma r\|_{L^2}\left\|\frac 1{(\sqrt{s}+\alpha|\sigma|)^2}\right\|_{L^2}+C\left\|\frac \alpha{(\sqrt{s}+\alpha|\sigma|)^2}\right\|_{L^1}.$$

\subsection{Estimates on $b$} We recall the expression of 
$$b(s,\sigma)=2 \ds H(s,\sigma) \partial_\sigma \left(\frac{1}{1+\psi(\alpha |\sigma|)}\right)+H(s,\sigma)\partial_\sigma^2\left(\frac{1}{1+\psi(\alpha |\sigma|)}\right).$$
We notice first that for $|\sigma|\leq 1/\alpha$  we have $\partial_\sigma \psi(\alpha |\sigma|)=0$ and therefore $b(s,\sigma)=0$, and \eqref{bL1}, \eqref{gradbL2I} and \eqref{gradbL1} follow. 

\medskip

When $|\alpha \sigma| > 2$, we have $\psi(\alpha |\sigma|) = \alpha |\sigma|$. 
We thus obtain 
$$
|b(s,\sigma)| \leq  2 | \ds H(s,\sigma)| \frac{\alpha}{(1 + \alpha |\sigma|)^2}
+ 2  |H(s,\sigma)| \frac{\alpha^2}{(1 + \alpha |\sigma|)^3}.
$$
Using \eqref{ineq:upper-bound-2} and the fact that $|\partial_\sigma H(s,\sigma)|\leq C\alpha$, we obtain the bound 
$$
|b(s,\sigma)| \leq C \frac{\alpha^2}{(1+\psi(\alpha |\sigma|))^2}. 
$$

On the other hand, when $|\alpha \sigma| \in [1/2,2]$, we have 
$$
|b(s,\sigma)| \leq  C (\alpha | \ds H(s,\sigma)|
+ \alpha^2  |H(s,\sigma)| ) \leq C \alpha^2 
$$
using again \eqref{ineq:upper-bound-2}. 
By integrating we obtain \eqref{bL2} in view of \eqref{estint2}. 

\medskip

Using the same arguments, we have for $|\alpha \sigma| > 2$ 
\begin{equation*}
\begin{split}
|\partial_\sigma b(s,\sigma)| \leq & C| \ds^2 H(s,\sigma)| \frac{\alpha}{(1 + \alpha |\sigma|)^2} \\&+ C \left( | \ds H(s,\sigma)| \frac{\alpha^2}{(1 + \alpha |\sigma|)^3} + |H(s,\sigma)| \frac{\alpha^3}{(1 + \alpha |\sigma|)^4} \right). 
\end{split}
\end{equation*}
Since $|\partial_\sigma^2 H(s,\sigma)|\leq C \alpha s^{-\frac 12}$, we obtain for $|\alpha \sigma | > 2$, using again the bound \eqref{ineq:upper-bound-2},
$$|\partial_\sigma b(s,\sigma)|\leq C\frac{\alpha^2 s^{-\frac 12}}{(1+\psi(\alpha |\sigma|))^2}+ C\frac{\alpha^3}{(1+\psi(\alpha |\sigma|))^3}.
$$
On the other hand, when $|\alpha \sigma| \in [1/2,2]$ we obtain 
\begin{equation*}
\begin{split}
|\partial_\sigma b(s,\sigma)| &\leq  C\left( \alpha | \ds^2 H(t,\sigma)| + \alpha^2 | \ds H(t,\sigma)| + \alpha^3 |H(t,\sigma)|\right) \\ &\leq C (\alpha^2 s^{-\frac12}+\alpha^3). 
\end{split}
\end{equation*}
We obtain the inequality \eqref{gradbL2} by combining the two latter estimates and using \eqref{estint2}.

\subsection{Estimates on $a(r_1)-a(r_2)$}
Now we turn to the estimate on $c=a(r_1)-a(r_2)$, given by
$$c=\frac{\re(r_2-r_1)}{\left(\re (r_1)+\frac{\re(H)}{1+\psi(\alpha |\sigma|)}\right)\left(\re (r_2)+\frac{\re(H)}{1+\psi(\alpha |\sigma|)}\right)}.$$
In view of \eqref{ineq:near-1}, \eqref{ineq:far-1}, \eqref{ineq:near-2} and \eqref{ineq:far-2} we have for all $ (s,\sigma)\in [0,t_0]\times I$
\begin{equation}\label{cLinftynear}
|c(s,\sigma)|\leq C\frac{|(r_1-r_2)(s,\sigma)|}{(\sqrt{s}+\alpha |\sigma|)^2},
\end{equation}
and for $ (s,\sigma)\in [0,t_0]\times \R\setminus I$
\begin{equation}\label{cLinftyfar}
|c(s,\sigma)|\leq C|(r_1-r_2)(s,\sigma)|.
\end{equation}
Therefore  \eqref{cLinftyfar} and \eqref{schumann} imply
\begin{equation*}
\begin{split}
\|c(s)\|_{L^2}&\leq C\|(r_1-r_2)(s)\|_{L^\infty(I)} \left\|\frac{1}{(\sqrt{s}+\alpha |\sigma|)^2} \right\|_{L^2(I)} +  \|(r_1-r_2)(s)\|_{L^2}\\
&\leq C\|(r_1-r_2)(s)\|_{L^\infty(I)} (\alpha^{-\frac12}s^{-\frac34})  +  \|(r_1-r_2)(s)\|_{L^2}
\end{split}
\end{equation*}
which yields \eqref{cL2}. Also,
\begin{equation*}\begin{split}
\|c(s)\|_{L^1(-\frac 1\alpha,\frac 1\alpha)}&\leq C\|(r_1-r_2)(s)\|_{L^2}\left(\left\|\frac{1}{(\sqrt{s}+\alpha |\sigma|)^2}\right\|_{L^2(I)}+\alpha^{-\frac 12}\right)\\
&\leq C\|(r_1-r_2)(s)\|_{L^2}\left(\alpha^{-\frac 12}s^{-\frac 34}+\alpha^{-\frac 12}\right),\end{split}\end{equation*}
so we have obtained  \eqref{cL1}. 

\medskip
We then express the  derivative of $c$ as
\begin{equation*}
\begin{split}
\partial_\sigma c =&\frac{\partial_\sigma \re(r_2-r_1)}{\left(\re (r_1)+\frac{\re(H)}{1+\psi(\alpha |\sigma|)}\right)\left(\re (r_2)+\frac{\re(H)}{1+\psi(\alpha |\sigma|)}\right)}\\
&-  \frac{\re(r_2-r_1)\left(\re (\partial_\sigma r_1)+\partial_\sigma \left(\frac{\re(H)}{1+\psi(\alpha |\sigma|)}\right)\right)}{\left(\re (r_1)+\frac{\re(H)}{1+\psi(\alpha |\sigma|)}\right)^2\left(\re (r_2)+\frac{\re(H)}{1+\psi(\alpha |\sigma|)}\right)}\\
&-  \frac{\re(r_2-r_1)\left(\re (\partial_\sigma r_2)+\partial_\sigma \left(\frac{\re(H)}{1+\psi(\alpha |\sigma|)}\right)\right)}{\left(\re (r_1)+\frac{\re(H)}{1+\psi(\alpha |\sigma|)}\right)\left(\re (r_2)+\frac{\re(H)}{1+\psi(\alpha |\sigma|)}\right)^2}.
\end{split}
\end{equation*}

For $ (s,\sigma)\in [0,t_0]\times I$ we  have using the bound \eqref{chili} on the derivative of $H$  and  using \eqref{ineq:near-2} 
\begin{equation}\label{gradcLinftynear}\begin{split}
|\partial_\sigma c(s,\sigma)|&\leq C\frac{|\partial_\sigma(r_1-r_2)|}{(\sqrt{s}+\alpha |\sigma|)^2}\\
&+C\frac{(|\partial_\sigma r_1|+|\partial_\sigma r_2|+\alpha)|r_1-r_2|}{(\sqrt{s}+\alpha |\sigma|)^3}.\end{split}\end{equation}
From this inequality, we infer  that
\begin{equation*}
\begin{split}
&\|\partial_\sigma c(s)\|_{L^2( I)}\leq C\frac{\|\partial_\sigma(r_1-r_2)(s)\|_{L^2}}{s}\\
&+C\frac{(\|\partial_\sigma r_1\|_{L^2}+\|\partial_\sigma r_2\|_{L^2})\|(r_1-r_2)(s)\|_{L^\infty(I)}+\alpha\|(r_1-r_2)(s)\|_{L^2}}{s^\frac 32}\\
&\leq  C\|r_1-r_2\|_X (s^{\gamma-1}+\alpha s^{-\frac 34}),\end{split}\end{equation*}
i.e. \eqref{gradcL2I}.

Using also \eqref{estint1} and \eqref{schumann},
\begin{equation*}
\begin{split}
&\|\partial_\sigma c(s)\|_{L^1( I)}\leq C\|\partial_\sigma(r_1-r_2)(s)\|_{L^2} \left\|\frac{1}{(\sqrt{s}+\alpha |\sigma|)^2}\right\|_{L^2(I)}\\
&+C(\|\partial_\sigma r_1\|_{L^2}+\|\partial_\sigma r_2\|_{L^2})\|(r_1-r_2)(s)\|_{L^\infty(I)}\left\|\frac 1{(\sqrt{s}+\alpha |\sigma|)^3}\right\|_{L^2(I)}\\
&+C\|(r_1-r_2)(s)\|_{L^\infty(I)}\left\|\frac {\alpha}{(\sqrt{s}+\alpha |\sigma|)^3}\right\|_{L^1(I)}\\
&\leq  C\|r_1-r_2\|_X (\alpha^{-\frac12}s^{\gamma-\frac 34}+  s^{-\frac 12}),
\end{split}
\end{equation*}
which yields  \eqref{gradcL1}. 

\medskip

When $(s,\sigma) \in [0,t_0] \times \mathbb{R} \setminus I$, we have using \eqref{ineq:far-2}  and the bound \eqref{chili} for the derivative of $H$, 
\begin{equation}\label{gradcLinftyfar}\begin{split}
|\partial_\sigma c(s,\sigma)|&\leq C|\partial_\sigma(r_1-r_2)(s,\sigma)|\\
&+C\left(|\partial_\sigma r_1|+|\partial_\sigma(r_2)|+\frac{\alpha}{1+\psi(\alpha |\sigma|)}\right)|(r_1-r_2)(s,\sigma)|.
\end{split}\end{equation}
Hence we obtain 
\begin{equation*}\begin{split}
&\|\partial_\sigma c(s)\|_{L^2(\mathbb R\setminus I)}\leq C\|\partial_\sigma(r_1-r_2)(s)\|_{L^2}\\
&+C\Big(\|\partial_\sigma r_1\|_{L^2}+\|\partial_\sigma(r_2)\|_{L^2}+\left\|\frac{\alpha}{1+\psi(\alpha |\sigma|)}\right\|_{L^2(\mathbb{R}\setminus I) }\Big)\|(r_1-r_2)(s)\|_{L^\infty}.\end{split}\end{equation*}
We use \eqref{ineq:upper-r}  and \eqref{estint2} together with the Sobolev embedding $H^1(\R)\subset L^\infty(\R)$ to get
$$\|\partial_\sigma c(s)\|_{L^2(\mathbb R\setminus I)}\leq C\|\partial_\sigma(r_1-r_2)(s)\|_{L^2}+C\|(r_1-r_2)(s)\|_{L^2}^\frac 12\|\partial_\sigma(r_1-r_2)(s)\|_{L^2}^\frac 12(s^\gamma+\alpha^\frac 12),$$
and hence \eqref{gradcL2} follows.

\subsection{Proof of Lemma \ref{est-pol}}

\mbox{}
\medskip

 \textbullet\: \textbf{Estimates for $\tilde{a}(r)$.}
We recall that
$$\tilde{a}(r)=- \frac{1}{\overline{r}+\frac{\overline{H}}{1+\psi(\alpha |\sigma|)}}
+{r}+\frac{{H}}{1+\psi(\alpha |\sigma|)}+\frac{1}{(1+\psi(\alpha |\sigma|))\overline{H}}.$$

For $\sigma \in I$ we have $\psi(\alpha |\sigma|)=0$, so
$$\tilde{a}(r)=-\frac{1}{\overline{r}+\overline{H}}+\frac{1}{\overline{H}}+r+H
=\frac{  \overline{r}}{(\overline{r}+\overline{H})\overline{H}}+r+H,$$
hence applying \eqref{ineq:near-1}, \eqref{ineq:near-2} and \eqref{ineq:upper-bound-2} we get \eqref{ineq:a-1-pol}.

\medskip

For $\sigma \in \R\setminus I$ we have
\begin{equation*}
\begin{split}
\tilde{a}(r)&=\frac{1}{(1+\psi(\alpha |\sigma|))\overline{H}}- \frac{1}{\overline{r}+\frac{\overline{H}}{1+\psi(\alpha |\sigma|)}}
+\frac{H}{1+\psi(\alpha |\sigma|)}+r\\
&=\frac{1}{(1+\psi(\alpha |\sigma|))\overline{H}}
+\frac{\overline{r} H}
{(1+\psi(\alpha |\sigma|))(\overline{r}+\frac{\overline{H} }{1+\psi(\alpha |\sigma|))})}\\&+
\frac{1}{\overline{r}+\frac{\overline{H}}{1+\psi(\alpha |\sigma|)}}\left( \frac{|H|^2}{(1+\psi(\alpha |\sigma|))^2}-1 \right)+r,
\end{split}
\end{equation*}
therefore by \eqref{ineq:far-1}, \eqref{ineq:far-2} and \eqref{ineq:upper-bound-2}, 
\begin{equation}\label{ineq:a-3-pol}
\begin{split}
|\tilde{a}(r)(s,\sigma)|
&\leq \frac{C}{(1+\psi(\alpha |\sigma|))^2}+C
\frac{|r(s,\sigma)(\sqrt{s}+\alpha|\sigma|)}{1+\psi(\alpha |\sigma|)}\\
&+C\left| \frac{|H(s,\sigma)|^2}{(1+\psi(\alpha |\sigma|))^2}-1 \right|+|r(s,\sigma)|\\
&\leq \frac{C}{(1+\psi(\alpha |\sigma|))^2}+C|r(s,\sigma)|+C\left| \frac{|H(s,\sigma)|^2}{(1+\psi(\alpha |\sigma|))^2}-1 \right|,
\end{split}
\end{equation}
where we have used the fact that $x\mapsto (1+x)/(1+\psi(x))$ is bounded on $\R_+$.
We now claim that
\begin{equation}
\label{ineq:chili}
\left| \frac{|H(s,\sigma)|^2}{(1+\psi(\alpha |\sigma|))^2}-1 \right|\leq C\left(\frac{\alpha^2}{(1+\psi(\alpha |\sigma|))^2}+\frac{\alpha}{1+\psi(\alpha |\sigma|)}\right),\quad \forall \sigma\in \R\setminus I,
\end{equation}
from which \eqref{ineq:a-2-pol} follows.
In order to prove \eqref{ineq:chili}, we recall that by \eqref{dev:H}, $$H(s,\sigma)=\alpha |\sigma|+R(s,\alpha),\quad  |R(s,\sigma)|\leq \sqrt{s}(1+\alpha)\leq 2\alpha.$$
Therefore
\begin{equation*}
\begin{split}
\left| \frac{|H(s,\sigma)|^2}{(1+\psi(\alpha |\sigma|))^2}-1 \right|&\leq 
\frac{C}{(1+\psi(\alpha |\sigma|))^2}\left(|(\alpha |\sigma|)^2-\psi(\alpha |\sigma|)^2|+\alpha^2+1+
 \psi(\alpha |\sigma|)+\alpha^2|\sigma|\right).
\end{split}
\end{equation*}
Since $ \psi(\alpha |\sigma|)\leq \alpha |\sigma|$ and $(\alpha |\sigma|)^2-\psi(\alpha |\sigma|)^2=\mathds{1}_{\alpha |\sigma|\leq 2}|(\alpha |\sigma|)^2-\psi(\alpha |\sigma|)^2|$ we infer that
\begin{equation*}\begin{split}
\left| \frac{|H(s,\sigma)|^2}{(1+\psi(\alpha |\sigma|))^2}-1 \right|&\leq 
\frac{C}{(1+\psi(\alpha |\sigma|))^2}(\alpha^2+\alpha^2 \sigma)\\
&\leq 
C\left(\frac{\alpha^2}{(1+\psi(\alpha |\sigma|))^2}
+\frac{\alpha}{1+\psi(\alpha |\sigma|)}\frac{\alpha |\sigma|}{1+\psi(\alpha |\sigma|)}\right)\\
& \leq
C\left(\frac{\alpha^2}{(1+\psi(\alpha |\sigma|))^2}+\frac{\alpha}{1+\psi(\alpha |\sigma|)}\right),\quad \forall \sigma\in \R\setminus I,
\end{split}
\end{equation*}
where we have used the fact that $x\mapsto x/(1+\psi(x))$ is bounded on $\R_+$.
\medskip

Using the pointwise estimates \eqref{ineq:a-1-pol} and \eqref{ineq:a-2-pol}, we establish the estimate \eqref{aL2-pol} exactly as \eqref{aL2}. Moreover, we show that \eqref{aL1I} holds true for $\tilde{a}(r)$ since  on the set $I$, $\tilde{a}(r)(s)$ satisfies the same pointwise estimate as $a(r)(s)$, while for $1/2\leq \alpha|\sigma| \leq 1$ we have $\sigma \in \R\setminus I$, hence we obtain by \eqref{ineq:a-3-pol} and \eqref{ineq:upper-r}
\begin{equation*}
\begin{split}
|\tilde{a}(r)(s,\sigma)|
&\leq C+C\left| \frac{|H(s,\sigma)|^2}{(1+\psi(\alpha |\sigma|))^2}-1 \right|\leq C.
\end{split}
\end{equation*}

We next estimate the derivatives. We have
\begin{equation*}\begin{split}
\partial_\sigma \tilde{a}(r)&
= \frac{\partial_\sigma \overline{r}+\partial_\sigma \left(\frac{\overline{H}}{1+\psi(\alpha|\sigma|)}\right)}
{\left( \overline{r}+\frac{\overline{H}}{1+\psi(\alpha|\sigma|)}\right)^2}
-\frac{1}{\overline{H}}\frac{\alpha \text{sgn}({\sigma})\psi'(\alpha |\sigma|)}{(1+\psi(\alpha |\sigma|))^2}- \frac{\partial_\sigma \overline{H}}{\overline{H}^2(1+\psi(\alpha |\sigma|))}\\
&+\partial_\sigma r + \frac{\partial_\sigma H}{1+\psi(\alpha |\sigma|)}- H\frac{\alpha \text{sgn}({\sigma})\psi'(\alpha |\sigma|)}{(1+\psi(\alpha |\sigma|))^2}.
\end{split}
\end{equation*}

Let $\sigma\in\mathbb R\setminus I$. By \eqref{ineq:far-1}, \eqref{ineq:far-2} and the inequality $|\partial_\sigma H(s,\sigma)|\leq 2\alpha$, we get
\begin{equation*}\begin{split}
|\partial_\sigma \tilde{a}(r)(s,\sigma)|&\leq C|\partial_\sigma r|+C
\left|\partial_\sigma \left(\frac{H}{1+\psi(\alpha|\sigma|)}\right)\right|
+\frac{ C\alpha }{1+\psi(\alpha|\sigma|)}+ \frac{\alpha |H|}{(1+\psi(\alpha|\sigma|))^2}\\
&\leq C|\partial_\sigma r|+C
\left|\partial_\sigma \left(\frac{H}{1+\psi(\alpha|\sigma|)}\right)\right|
+\frac{ C\alpha }{1+\psi(\alpha|\sigma|)}\\
&\leq C|\partial_\sigma r|+
+\frac{ C\alpha }{1+\psi(\alpha|\sigma|)},
\end{split}
\end{equation*}where we have used  \eqref{chili} in the last inequality. So we obtain \eqref{gradaL2}.

\medskip

For $\sigma\in I$ we have $\psi(\alpha|\sigma|)=0$. Byy \eqref{ineq:near-1} and \eqref{ineq:near-2}, we infer that
\begin{equation*}
|\partial_\sigma \tilde{a}(r)(s,\sigma)|\leq C\frac{|\partial_\sigma r|+\alpha}{(\sqrt{s}+\alpha|\sigma|)^2}\leq  C\left(\frac{|\partial_\sigma r|}{s}+\frac\alpha {(\sqrt{s}+\alpha|\sigma|)^2}\right),
\end{equation*}which is the pointwise estimate \eqref{gradaLinftybis} for $\partial_\sigma a(r)$.
Hence we obtain  \eqref{gradaL2I} and \eqref{gradaL1}.

\bigskip

\textbullet \: \textbf{Estimates for $\tilde{a}(r_1)-\tilde{a}(r_2)$.} We have
$$\tilde{c}=\tilde{a}(r_1)-\tilde{a}(r_2)=
\frac{\overline{r_1-r_2}}{\left(\overline {r_1}+\frac{\overline{H}}{1+\psi(\alpha |\sigma|)}\right)\left(\overline{r_2}+\frac{\overline{H}}{1+\psi(\alpha |\sigma|)}\right)}+r_1-r_2.$$
In view of \eqref{ineq:near-1}, \eqref{ineq:far-1}, \eqref{ineq:near-2} and \eqref{ineq:far-2} we have for all $ (s,\sigma)\in [0,t_0]\times I$
\begin{equation*}\label{cLinftynear-pol}
|\tilde{c}(s,\sigma)|\leq C\left(\frac{|(r_1-r_2)(s,\sigma)|}{(\sqrt{s}+\alpha |\sigma|)^2}+|(r_1-r_2)(s,\sigma)\right)|\leq C|\frac{|(r_1-r_2)(s,\sigma)|}{(\sqrt{s}+\alpha |\sigma|)^2},
\end{equation*}
which yields the same pointwise estimate as \eqref{cLinftynear} for $c$, and for $ (s,\sigma)\in [0,t_0]\times \R\setminus I$
\begin{equation*}\label{cLinftyfar-pol}
|\tilde{c}(s,\sigma)|\leq C|(r_1-r_2)(s,\sigma)|,
\end{equation*}
which yields the same estimate as \eqref{cLinftyfar}. Hence we immediately obtain \eqref{cL2} and \eqref{cL1}.

\medskip

Next, we compute the derivative of $\tilde{c}$,
\begin{equation*}
\begin{split}
\partial_\sigma \tilde{c} =& \frac{\partial_\sigma (\overline{r_1}-\overline{r_2})}{\left(\overline{r_1}+\frac{\overline{H}}{1+\psi(\alpha |\sigma|)}\right)\left(\overline{r_2}+\frac{\overline{H}}{1+\psi(\alpha |\sigma|)}\right)}-  \frac{(\overline{r_1}-\overline{r_2})\left(\partial_\sigma \overline{r_1}+\partial_\sigma \left(\frac{\overline{H}}{1+\psi(\alpha |\sigma|)}\right)\right)}{\left(\overline{r_1}+\frac{\overline{H}}{1+\psi(\alpha |\sigma|)}\right)^2\left(\overline{r_2}+\frac{\overline{H}}{1+\psi(\alpha |\sigma|)}\right)}\\
&-  \frac{(\overline{r_1}-\overline{r_2})\left(\partial_\sigma \overline{r_2}+\partial_\sigma \left(\frac{\overline{H}}{1+\psi(\alpha |\sigma|)}\right)\right)}{\left(\overline{r_1}+\frac{\overline{H}}{1+\psi(\alpha |\sigma|)}\right)\left(\overline{r_2}+\frac{\overline{H}}{1+\psi(\alpha |\sigma|)}\right)^2}+ \partial_\sigma( r_1- r_2).
\end{split}
\end{equation*}
Therefore we infer that  $\tilde{c}$ satisfies the estimates \eqref{gradcLinftynear} and \eqref{gradcLinftyfar}, which leads to  \eqref{gradcL2}, \eqref{gradcL2I} and \eqref{gradcL1}.

\end{document}